\documentclass{amsart}
\usepackage{amsmath,amsthm,amsfonts,amssymb,latexsym,mathrsfs,graphicx}
\usepackage{pgf,tikz}
\usepackage{ulem}
\usetikzlibrary{arrows}

\usepackage{hyperref}

\usepackage{enumerate}
\usepackage[shortlabels]{enumitem}
\usepackage{color}

\usepackage{pstricks-add}
\usepackage{graphicx}
\usepackage{pgf,tikz,pgfplots}
\usepackage{mathrsfs}
\usetikzlibrary{arrows}
\definecolor{ududff}{rgb}{0.30196078431372547,0.30196078431372547,1}

\newtheorem{theorem}{Theorem}[section]
\newtheorem{corollary}[theorem]{Corollary}
\newtheorem{lemma}[theorem]{Lemma}
\newtheorem{proposition}[theorem]{Proposition}
\newtheorem{thm}{Theorem}

\newtheorem{prop}[thm]{Proposition}

\theoremstyle{definition}

\newtheorem{example}[theorem]{Example}

\def\cent#1#2{{\bf C}_{#1}(#2)}

\def\syl#1#2{{\rm Syl}_#1(#2)}

\def\oh#1#2{{\bf O}_{#1}(#2)}
\def\zent#1{{\bf Z}(#1)}

\def\cent#1#2{{\bf C}_{#1}(#2)}

\def\syl#1#2{{\rm Syl}_#1(#2)}

\def\norm#1#2{{\bf N}_{#1}(#2)}
\def\oh#1#2{{\bf O}_{#1}(#2)}

\def\zent#1{{\bf Z}(#1)}

\def\cent#1#2{{\bf C}_{#1}(#2)}
\def\syl#1#2{{\rm Syl}_#1(#2)}
\def\oh#1#2{{\bf O}_{#1}(#2)}

\def\zent#1{{\bf Z}(#1)}

\def\norm#1#2{{\bf N}_{#1}(#2)}

\mathchardef\coso="2023

\begin{document}

\title{On non self-normalizing subgroups}
\author{\centerline{
    Mariagrazia Bianchi, Rachel D. Camina, 
    Mark L. Lewis, Emanuele Pacifici and Lucia Sanus
}}

\address{M. Bianchi, Dipartimento di Matematica F. Enriques, Universit\`a degli Studi di Milano, via Saldini 50, 20133 Milano, Italy}
\email{mariagrazia.bianchi@unimi.it}

\address{R. D. Camina, Fitzwilliam College, Cambridge, CB3 0DG, United Kingdom}
\email{rdc26@cam.ac.uk}

\address{M. L. Lewis, Department of Mathematical Sciences, Kent State University, Kent, OH 44242, USA}
\email{lewis@math.kent.edu}

\address{E. Pacifici, Dipartimento di Matematica e Informatica U. Dini, Universit\`a degli Studi di Firenze, Viale Morgagni 67/A, 50134 Firenze, Italy.}
\email{emanuele.pacifici@unifi.it}

\address{L. Sanus, Departament de Matem\`atiques, 
Universitat de Val\`encia,
46100 Burjassot, Val\`encia, Spain. }
\email{lucia.sanus@uv.es}
\renewcommand{\shortauthors}{Bianchi et al.}


\begin{abstract}
Let $n$ be a non negative integer, and define ${\mathcal D}_n$ to be the family of all finite groups having precisely $n$ conjugacy classes of nontrivial subgroups that are not self-normalizing.  We are interested in studying the behavior of ${\mathcal D}_n$ and its interplay with solvability and nilpotency.  We first show that if $G$ belongs to ${\mathcal D}_n$ with $n \le 3$, then $G$ is solvable of derived length at most $2$.  We also show that ${\rm A}_5$ is the unique nonsolvable group in ${\mathcal D}_4$, and that ${\rm SL}_2(3)$ is the unique solvable group in ${\mathcal D}_4$ whose derived length is larger than $2$.  For a group $G$, we define ${\mathcal D}(G)$ to be the number of conjugacy classes of nontrivial subgroups that are not self-normalizing.  We determine the relationship between ${\mathcal D}(H \times K)$ and ${\mathcal D} (H)$ and ${\mathcal D}(K)$.  We show that if $G$ is nilpotent and lies in ${\mathcal D}_n$, then $G$ has nilpotency class at most $n/2$ and its derived length is at most $\log_2 (n/2) + 1$.  We consider ${\mathcal D}_n$ for several classes of Frobenius groups, and we use this classification to classify the groups in ${\mathcal D}_0$, ${\mathcal D}_1$, ${\mathcal D}_2$, and ${\mathcal D}_3$.  Finally, we show that if $G$ is solvable and lies in ${\mathcal D}_n$ with $n \ge 3$, then $G$ has derived length at most the minimum of $n-1$ and $3 \log_2 (n+1) + 9$. 
\end{abstract}
\thanks{ The first and fourth authors are partially supported by INdAM-GNSAGA and by the European Union-Next Generation EU, Missione 4 Componente 1, CUP B53D23009410006, PRIN 2022 2022PSTWLB - Group Theory and Applications. The research of the fifth  author is supported by Ministerio de Ciencia e Innovaci\'on (PID2022-137612NB-I00 funded by MCIN/AEI/10.13039/501100011033 and `ERDF A way of making Europe') and Generalitat Valenciana CIAICO/2021/163.}

\subjclass[2010]{Primary 20E07, 20E34; Secondary 20D10, 20D15}

\keywords{non self-normalizing, solvable groups, derived length}

\maketitle
\section{Introduction}

In this paper we focus on finite groups.  We take $n$ to be a non negative integer, and we define ${\mathcal D}_n$ to be the family of all finite groups having precisely $n$ conjugacy classes of nontrivial subgroups that are $\mathit{not}$ self-normalizing.  We are particularly interested in studying the behavior of ${\mathcal D}_n$ and its interplay with solvability and nilpotency.   However, before we begin to look at nilpotency or solvability it makes sense to classify ${\mathcal D}_n$ for some small values of $n$.  When $n = 0$, we have:

\begin{prop} \label {main D_0}
A group is in ${\mathcal D}_0$ if and only if it has prime order.
\end{prop}

When $n = 1$, we see that:

\begin{prop} \label{main D_1} 
The only abelian groups in ${\mathcal D}_1$ are cyclic groups of order $p^2$, where $p$ is a prime number.  Meanwhile, the only nonabelian groups in  ${\mathcal D}_1$ are Frobenius groups of order $pq$,  where $p$ and $q$ are prime numbers. 
\end{prop}

In Section \ref{sect:D_3} we will in fact classify the groups in ${\mathcal D}_2$ and ${\mathcal D}_3$, but these are complicated enough that we do not include them in the introduction.  

As is standard, we begin the study of nilpotent groups by looking at $p$-groups.  Note that when $n = 0$ and when $n=1$, we see that the $p$-groups lying in ${\mathcal D}_n$ are precisely the cyclic groups of order $p^{n+1}$.  In fact, it is not difficult to see that the cyclic groups of order $p^{n+1}$ always lie in ${\mathcal D}_n$.  To get a uniform conclusion for this next result, we need to consider the noncyclic $p$-groups, and so, it is enough to assume that $n > 1$.  Thus, for $p$-groups, we have the following:

\begin{thm} \label{main p-group}
Let $G$ be a $p$-group belonging to ${\mathcal D}_n$ with $n > 1$ for some prime $p$.  Then $G$ has nilpotency class at most $n/2$ and derived length at most $\displaystyle {\rm log}_2(n/2) +1$.
\end{thm}

Using our result for $p$-groups, we are able to extend our result for nilpotent groups.   For nilpotent groups, we prove:

\begin{thm} \label{main nilpotent}
Let $G$ be a noncyclic nilpotent group with order divisible by exactly $k$ distinct primes.  If $G$ lies in ${\mathcal D}_n$, then $G$ has nilpotency class at most $(n + 2 - 2^k)/2^k$.  In particular, the derived length of $G$ is at most $\displaystyle {\rm log}_2((n + 2 - 2^k)/{2^k}) + 1$.
\end{thm}

We now turn our attention to solvability.  We begin by showing the solvability of the groups in ${\mathcal D}_n$ when $n$ is ``small''.  We also compute the derived lengths of the groups when $n$ is ``small''.
	
\begin{thm} \label{main solvable le 3}
Let $G$ be a group belonging to $D_n$ for $n \leq 3$.  Then $G$ is solvable with derived length at most $2$.
\end{thm}

We will see that for $n \ge 4$, we no longer have that the groups in ${\mathcal D}_n$ are necessarily solvable.  
However, we have computed the derived lengths of many solvable groups using the computer algebra systems Magma \cite{magma} and GAP \cite{gap}.  From those calculations, it seems likely that when $G$ is solvable and $G \in {\mathcal D}_n$, then the derived length of $G$ is actually bound by a logarithmic function in $n$.   As a first step in this direction, we have the following result, although the bound obtained in this theorem seems to be much larger than those that occur for actual groups.   Hence, we expect that this bound can be improved. 

\begin{thm} \label{main dl log}
Let $n \ge 3$ be an integer.  If $G \in \mathcal{D}_n$ is solvable, then the derived length of $G$ is less than or equal to the minimum of $n-1$ and $3 \log_2 (n+1)+9$.	
\end{thm}

As we stated above, we will classify the groups in ${\mathcal D}_2$ and ${\mathcal D}_3$.  We will see that as $n$ grows, the number of classes of groups grows, and we believe that for $n = 4$, the number of classes of groups in ${\mathcal D}_4$ is beyond what is profitable to classify.  It is not difficult to show that ${\rm A}_5$ lies in ${\mathcal D}_4$.  We next show that ${\rm A}_5$ is the only nonsolvable group in ${\mathcal D}_4$.

\begin{thm}\label{main A5} 
Let $G$ be a nonsolvable group. If $G\in\mathcal{D}_4$, then $G\cong\textnormal{A}_5$.
\end{thm}

Also, we can consider the solvable groups in ${\mathcal D}_4$.  It is not difficult to see that ${\rm SL}_2 (3)$ lies in ${\mathcal D}_4$ and has derived length $3$.  We show that this is the only solvable group in ${\mathcal D}_4$ whose derived length is $3$ and others have derived length at most $2$. 

\begin{thm} \label{main D=4,dl} 
Suppose $G$ lies in ${\mathcal D}_4$ and $G \neq {\rm A}_5$. Then $G$ is solvable and either has derived length at most $2$ or is isomorphic to ${\rm SL}_2 (3)$. 
\end{thm}

Some of the work on this project was done while the first, second, third and fifth authors were visiting Universit\`a degli Studi di Firenze in March and April 2024 and while the third author was visiting Fitzwilliam College at the University of Cambridge also in April 2024.  They would like to thank those universities for their hospitality.   The authors would also like to thank Gareth Jones for his comments and for sharing his preprint \cite{JCount} with them.

\section{background}

In this section we provide some background and begin to think about the problem. 
The question of classifying groups, both finite and infinite, in which all subgroups of a certain type are self-centralizing or self-normalizing has generated much interest. Note, a subgroup is called self-centralizing if it contains its centralizer. In particular, a self-normalizing subgroup is self-centralizing.
For example,  
G. Giordano in 1972, considered groups in which all non-normal subgroups are self-normalizing. Such groups are
 $T$-groups (i.e. groups in which normality is a transitive relation on the set of all subgroups) and they are periodic or abelian,
see \cite{Gi}. More recently, 
 finite groups in which all noncyclic subgroups are self-centralizing have been classified, see \cite[Theorem 2.1]{parker}. 
Finite groups in which all noncyclic subgroups are self-normalizing have also been classified, see \cite[Theorem 9.7]{DN}.
In \cite{maj}, which was the original motivation for this paper, the authors focus on non self-centralizing cyclic subgroups.
For an overview of results on this topic see \cite{DN}.


In this article we count (up to conjugacy) nontrivial subgroups of a group that are not self-normalizing and consider groups when this number is small. When looking for non self-normalizing subgroups, the first ones to come to mind are the proper normal subgroups.  And quickly one realizes that the proper subnormal subgroups will also be non self-normalizing. Thus
for nilpotent groups, which satisfy the condition that every proper subgroup is subnormal \cite[Lemma 2.1]{I}, we are counting conjugacy classes of proper nontrivial subgroups, and for abelian groups we are simply counting the number of proper nontrivial subgroups. The property that every proper subgroup
is non self-normalizing characterizes nilpotent groups, see \cite[Theorem 1.26]{I}, and thus for groups that are not nilpotent our count is more subtle.

If $P$ is a Sylow $p$-subgroup of $G$, then it is known and not difficult to show that $\norm G P$ is self-normalizing (see Problem 1.B.3 in \cite{I}).  In fact, if $\norm G P \le H \le G$, then $H = \norm G H$ (see Problem 1.C.1 in \cite {I}) and so, when $P$ is not normal, this is a way to find self-normalizing subgroups of $G$.  On the other hand, any proper subgroup of $P$ will be proper in its normalizer in $P$, and so, no proper subgroup of a Sylow subgroup can ever be self-normalizing. 

Recall that if $P$ is a Sylow $p$-subgroup so that $P$ is abelian and either $P =\norm G P$ or $P$  complemented in $\norm G P$, then $P \le \zent{\norm GP}$.  This is the situation of the so-called Burnside's normal $p$-complement theorem.  We will use this enough times that we will state it here.  This can be found as Theorem 5.13 in \cite{I}.

\begin{theorem}\label{Burnside} 
Let $G$ be a finite group, $p$ a prime, and $P$ a Sylow $p$-subgroup of $G$.  If $P$ is contained in the center of $\norm G P$, then $G$ has a normal $p$-complement.
\end{theorem}

Groups whose Sylow subgroups are all cyclic are known as {\it Z-groups}.  The structure of $Z$-groups is very tight.  For example in Theorem 5.16 of \cite{I} the following is proved:

\begin{theorem} \label{z-group}
If $G$ is a $Z$-group, then $G/G'$ and $G'$ are cyclic and have coprime orders.
\end{theorem}

As previously mentioned we will consider ${\mathcal D}_n$ for small values of $n$. We will see that many of the groups in these sets are Frobenius groups,  thus, we provide some background on Frobenius groups.  One of the best treatments of Frobenius groups is found in \cite{H} in Section V.8.  For an English treatment of Frobenius groups, we recommend Section 6 of \cite{I}.

Let $H\neq 1$ be a proper subgroup of a group $G$.  We say $H$ is a {\it Frobenius complement} if $H \cap H^g = 1$ for all $g \in G \setminus H$.  We say $G$ is a {\it Frobenius group} if $G$ contains a subgroup $H$ that is a Frobenius complement.  We let $N = G \setminus \cup_{g \in G} (H \setminus \{1\})^g$.   Frobenius' theorem is that $N$ is a normal subgroup of $G$ (called the {\it Frobenius kernel} of $G$).  It is not difficult to see that $G = HN$ and $H \cap N = 1$.  Frobenius has shown that $(|H|, |N|) = 1$ and that  the Sylow subgroups of $H$ are either cyclic or generalized quaternion.  In particular, the Sylow subgroups for odd primes are cyclic.  J. G. Thompson in his dissertation proved that $N$ (the Frobenius kernel) is nilpotent. 

\section{Preliminary results}\label{sect:prelim}


The classification of groups in ${\mathcal D}_0$ is almost trivial, so we start with it.

\begin{proposition} \label {D_0}
A group is in ${\mathcal D}_0$ if and only if it has prime order.
\end{proposition}

\begin{proof}
Since  groups of prime order do not have proper subgroups, they belong to ${\mathcal D}_0$.  Conversely, suppose that $G \in {\mathcal D}_0$.  This implies that $G$ has no proper nontrivial normal subgroups.  Hence, $G$ is simple.  If $G$ is a $p$-group, this proves the result.  We suppose $G$ is not a $p$-group, and let $p$ be a prime divisor of $|G|$.  We have seen that if $|P| > p$, then the proper nontrivial subgroups of $P$ yield non self-normalizing subgroups of $G$.  Thus, every Sylow subgroup of $G$ must have prime order.  In particular, every Sylow subgroup is cyclic and by Theorem \ref{z-group}, $G$ is solvable, so the theorem is proved. 
\end{proof}

We next study the behavior of $\mathcal{D}_n$ with respect to factor groups. In what follows, for a given group $G$, we denote by $\mathcal{D}(G)$ the number of conjugacy classes of nontrivial subgroups of $G$ that are not self-normalizing; also, if $N$ is a normal subgroup of $G$, we set ${\mathcal D}_G (N)$ to be the number of $G$-conjugacy classes of nontrivial and non self-normalizing subgroups of $G$ that are properly contained in $N$.  Observe that if $N = 1$ then ${\mathcal D}_G (N) = 0$, and if $N$ is central in $G$ then ${\mathcal D}_G (N) = {\mathcal D} (N)$.  

\begin{lemma} \label{rem} 
Let $G$ be a group and $N\neq 1$ a proper normal subgroup of $G$.  If $${\mathcal D} (G) - {\mathcal D}_G (N) = m,$$ then ${\mathcal D} (G/N) \le m-1$.   In particular, ${\mathcal D} (G/N) \le {\mathcal D} (G) - 1$.
\end{lemma}

\begin{proof} 
Let us denote by $\mathcal {S}$ the set of $G$-conjugacy classes $$\{H^G \mid H \text{ is a subgroup of } G \text{ such that } H \supseteq N\},$$ and by $\mathcal {T}$ the set of all the $G/N$-conjugacy classes $\{ (H/N)^{G/N} \mid H/N\text{ is a subgroup of } G/N\}$.  Also, denote by $\pi$ the canonical projection of $G$ onto $G/N$ and, for every $H^G \in \mathcal {S}$, set $f(H^G) = \pi(H)^{G/N}$.  Clearly, this defines a surjective map $f: \mathcal {S} \rightarrow \mathcal {T}$. 

Now, let $\mathcal {S}_0$ be the set consisting of the $G$-conjugacy classes in $\mathcal {S}$ of subgroups that are not self-normalizing in $G$, and let $\mathcal {T}_0$ be the set of $G/N$-conjugacy classes in $\mathcal {T}$ of subgroups that are not self-normalizing in $G/N$. Since, for every subgroup $H$ of $G$ containing $N$, we have $\norm {G/N}{H/N} = \norm G H/N$, the restriction of $f$ to $\mathcal {S}_0$ maps onto $\mathcal {T}_0$. In particular, we obtain $|\mathcal {T}_0| \leq |\mathcal {S}_0|$.  Note that $\mathcal {S}_0$ is clearly contained in the set of $G$-conjugacy classes of subgroups of $G$ which are not self-normalizing and not properly contained in $N$ (whose number is $m$ by our hypothesis); observe also that ${\mathcal D} (G/N) = |\mathcal {T}_0| - 1$ because ${\mathcal D} (G/N)$ counts all the classes in $\mathcal {T}_0$ except the trivial one.  It follows that ${\mathcal D} (G/N) \leq m - 1$, as claimed.
\end{proof}

As a consequence of the above lemma, if $N$ is a nontrivial proper normal subgroup of $G \in {\mathcal D}_n$, then $G/N$ lies in ${\mathcal D}_m$ for a suitable $m \leq n-1$.  
We now start analyzing the relationship between $\mathcal{D}_n$ and solvability. This connection will be investigated more thoroughly starting in Section \ref{sect:D_3}.

\begin{proposition} \label{metabelian} 
If $G$ is a group belonging to $\mathcal{D}_n$ for $n\leq 3$, then $G$ is solvable with derived length at most $2$.
\end{proposition}

\begin{proof}  
We can assume that $G$ is not abelian.  If every Sylow subgroup of $G$ is cyclic then, by Theorem~\ref{z-group}, 
$G/G'$ and $G'$ are cyclic groups of coprime orders; in particular, we can assume that there exists a prime $p$ such that $p^2$ divides $|G|$.  We consider among all the Sylow subgroups of $G$, the prime $p$ such that the order of the Sylow $p$-subgroup $P$ has the largest possible exponent. We write \(|P| = p^a\) with \(a \geq 2\).

We have that  $a$ cannot be larger than $4$, as otherwise $P$  (and therefore  $G$) would have subgroups of order $p$, $p^2$, $p^3$ and $p^4$ that are (pairwise nonconjugate and) not self-normalizing; on the other hand, if $|P|=p^4$, then we claim that $G=P$. In fact, if $q\neq p$ is a prime divisor of $|G|$ and $Q$ is a Sylow $q$-subgroup of $G$, then $Q$ must have order $q$ (as otherwise any subgroup of order $q$ would be normalized by a strictly larger $q$-subgroup of $G$); hence $Q$ is abelian and, since  it must be self-normalizing, Theorem~\ref{Burnside} yields the contradiction that $G$ has a normal $q$-complement. Now, $G=P$ must have a unique maximal subgroup, hence $G$ is cyclic.

Hence, either $a = 2$ or $a = 3$. Furthermore, if $q\neq p$ is a prime divisor of $|G|$, then $q^3$ cannot divide  $|G|$; therefore  the Sylow $q$-subgroups of $G$ are abelian.



Our first claim is that $|G|$ cannot be divisible by three distinct primes. Suppose, on the contrary, that there exist two distinct primes $q$  and $r$, different from $p$, both dividing $|G|$. Set  $Q\in\syl q G$, $R\in\syl r G$ and  $|Q|=q^{b}$, $|R|=r^{c}$ where $b,c \in \{1,2\}$.


Suppose that  $G$ has two subgroups, of order $p$ and $p^2$ respectively, that are not self-normalizing.   Then either $Q$ or $R$ must be self-normalizing. Assume this is $Q$: Theorem~\ref{Burnside} yields that $G$ has a normal $q$-complement, whence $R$ must be self-normalizing as well. But now $G$ also has a normal $r$-complement, a contradiction.

Now, suppose $a = 2$ and that $P$ is self-normalizing. By applying Theorem~\ref{Burnside} again, $G$ has a normal $p$-complement. This would imply that $G$ has two subgroups, one of order $p$ and the other being the $p$-complement, that are not self-normalizing. This situation leads to a contradiction, following the same reasoning as in the previous paragraph.

Our conclusion so far is that the order of $G$ is divisible by at most two distinct primes (note that this already implies the solvability of $G$ via Burnside's $p^aq^b$-Theorem: the character theory proof of this result can be found as Theorem 3.10 of \cite{Isaacs}), hence we write $|G|=p^aq^b$ and we can assume $3\geq a\geq 2 \geq b$. 

If $a=3$ and $b=2$, then $G$ has subgroups of order $p$, $p^2$ and $q$ that are not self-normalizing, thus forcing $Q$ to be self-normalizing: but $Q$ is abelian, hence again $P$ is normal in $G$ and we have a contradiction.

For $a= 3$, it  remains to consider the case when $b=1$ (so $G$ has subgroups of order $p$ and $p^2$ that are not self-normalizing). Since a direct check shows that the symmetric group $\textnormal{S}_4$ lies in $\mathcal{D}_7$, we have $G\not\cong\textnormal{S}_4$ and, as well-known, either $P$ or $Q$ are normal in $G$ (\cite[Theorem 1.32]{I}). Assume first $P\trianglelefteq G$: as $G'\subseteq P$, if $P$ is abelian then we get the desired conclusion, hence we can assume that $P$ is an extraspecial $p$-group of order $p^3$. But in this case $G$ has in fact two subgroups of order $p$ that are not conjugate in $G$ (namely, the center of $P$ and the subgroup generated by any non-central element of $P$ having order $p$), and we have a contradiction. On the other hand, assume $Q\trianglelefteq G$ and consider the normal subgroup $C=\cent GQ=Q\times\cent P Q$ of $G$. Observe that, since $P$ is not normal in $G$, $\cent P Q$ is a proper subgroup of $P$, whence $C$ is abelian; moreover, $G/C$ is also abelian, as it embeds in the automorphism group of $Q$. As a consequence, $G'\subseteq C$ is abelian and again we are done.

Next, assume $|G|=p^2q^2$. In this case, $G$ has subgroups of order $p$ and $q$ that are not self-normalizing, thus either $P$ or $Q$ must be self-normalizing; if this is $P$, then the abelian group $Q$ is normal in $G$, and $G/Q\cong P$ is also abelian, yielding the desired conclusion. Finally, if $|G|=p^2q$, then either $P$ or $Q$ are normal in $G$ 
and, as they are both abelian, the proof is complete.
\end{proof}

We close this section by describing ${\mathcal D}_1$.  We see that we have nonabelian examples that are solvable with derived length equal to $2$.  Note that when $G$ is a Frobenius group of order $pq$ with primes $p$ and $q$ where $q < p$, then $q$ divides $p-1$.

\begin{proposition} \label{D_1} 
The only abelian groups in ${\mathcal D}_1$ are cyclic groups of order $p^2$, where $p$ is a prime number.  Meanwhile, the only nonabelian groups in  ${\mathcal D}_1$ are Frobenius groups of order $pq$,  where $p$ and $q$ are prime numbers. 
\end{proposition}

\begin{proof}
Observe that if $G = C_{p^2}$ or $G$ is a Frobenius group of order $pq$ where $p$ and $q$ are primes, then $G \in {\mathcal D}_1$.  Conversely, assume $G$ is a group in ${\mathcal D}_1$. By Proposition \ref{metabelian}, we know that $G$ is solvable and has a derived length of either $1$ or $2$. If the derived length is $1$, then $G$ is abelian, and all its subgroups are normal in $G$. Since $G \in {\mathcal D}_1$, it follows that $G$ has only one proper, nontrivial subgroup, which implies that $G$ must be cyclic of order $p^2$ for some prime $p$. 
	
If $G$ has a derived length of 2, then $G'$ is the only proper, nontrivial, non self-normalizing subgroup of $G$. By Lemma \ref{rem}, $G/G'$ is in ${\mathcal D}_0$, and applying Proposition~\ref{D_0} shows that $G/G' \cong C_q$ for some prime $q$. Since $G'$ is abelian and the only non self-normalizing subgroup of $G$, which means $G'$ is also in ${\mathcal D}_0$. Thus, $G'$ is a cyclic group of prime order $p$, distinct from $q$, because if $p$ were equal to $q$, $G$ would be abelian. Therefore, $G$ is a Frobenius group of order $pq$.
\end{proof}


\section{Nilpotent groups}



When $n = 0$ or $1$, we saw in propositions~\ref{D_0} and \ref{D_1} that all of the nilpotent groups lying in ${\mathcal D}_n$ are cyclic of order $p$ or order $p^2$ where $p$ is a prime.  Furthermore, we noted in the introduction that the cyclic group of order $p^{n+1}$ will lie in ${\mathcal D}_n$.  For the purposes of computing nilpotency class and derived length, it suffices to look at the noncyclic $p$-groups, and so, we assume $n > 1$.  Since all of the subgroups of $G$ are non self-normalizing, we view this result as bounding the nilpotency class and the derived length of $G$ in terms of the number of conjugacy classes of proper nontrivial subgroups of $G$.  We would not be completely surprised if someone has studied this before, but we are not aware of any such research.

\begin{theorem} \label{nilpotent}
Let $p$ be a prime and $n>1$ an integer.  If $G$ is a $p$-group in ${\mathcal D}_n$, then $G$ has nilpotency class at most $n/2$ and the derived length of $G$ is at most ${\rm log}_2 (n/2) +1$.
\end{theorem}

\begin{proof}
Note that our assumption of $n$ being larger than $1$ implies that $G$ is not a cyclic group of order $p$ or $p^2$. Assume that $|G| = p^{m+1}$ for a positive integer $m$ and let $l$ be the nilpotency class of $G$.  We know that $l \le m$.  Note that every proper subgroup of $G$ is non self-normalizing.  Hence, $n$ is the number of conjugacy classes of nontrivial proper subgroups of $G$.  We claim that if $G$ is noncyclic, then it has at least $2m$ conjugacy classes of nontrivial proper subgroups, and we prove this by induction on the order of $G$. 
	
Let $Z$ be a central subgroup of $G$ having order $p$ and consider the factor group $G/Z$, which has order $p^m$. If $G/Z$ is cyclic then $G$ is abelian; moreover, taking an element $x$ of $G$ such that $xZ$ generates $G/Z$, the order of $x$ is a multiple of $p^m$. But since $G$ is noncyclic, the order of $x$ is in fact $p^m$ and we get $G=\langle x \rangle \times Z$. In this situation, it is not difficult to see that the number of nontrivial proper subgroups of $G$ is larger than $2m$, as desired. On the other hand, if $G/Z$ is noncyclic, then by induction it has at least $2m-2$ conjugacy classes of nontrivial proper subgroups; by the Correspondence Theorem (see Theorem 3.7 of \cite{I1}), these yield at least $2m-2$ conjugacy classes of proper subgroups of $G$ having order at least $p^2$. It is then enough to find two distinct conjugacy classes of subgroups of order $p$, and this can be done unless $G$ is a generalized quaternion group. But if this is the case, then $G/Z$ is a dihedral group of order $p^m$, so it has $3m-3$ conjugacy classes of nontrivial proper subgroups. Including $Z$ in the count, we have a total of $3m-2$ conjugacy classes of nontrivial proper subgroups of $G$, and since $3m-2\geq 2m$ our claim is proved.
	
This yields $n \ge 2m$, i.e., $n/2 \ge m \ge l$, and we are done in this case. Finally, our conclusion is straightforward if $G$ is cyclic (of order at least $p^3$).  The bound on the derived length follows from \cite[Corollary 4.13]{I}.
\end{proof} 

Before we turn to general nilpotent groups, we next consider the count of the number of non self-normalizing subgroups of a direct product.  Note that equality occurs only when both groups are nilpotent and have coprime orders.

\begin{lemma}\label{product}
Let $G = H \times K$ where $H$ and $K$ are both nontrivial.  Then we have $$\mathcal{D}(G) \ge (\mathcal{D}(H) + 2)(\mathcal{D}(K) + 2) - 2,$$  where equality occurs if and only if $H$ and $K$ have coprime orders and they are both nilpotent.
\end{lemma}

\begin{proof}
Let $X_1$ and $X_2$ be subgroups of $H$ and $Y_1$ and $Y_2$ be subgroups of $K$.  Then $X_1 \times Y_1$ is conjugate to $X_2 \times Y_2$ if and only if $X_1$ is conjugate to $X_2$ in $H$ and $Y_1$ is conjugate to $Y_2$ in $K$.  Notice that if $X$ is either $1_H$, $H$, or a proper, nontrivial, non self-normalizing subgroup of $H$, and $Y$ is either $1_K$, $K$, or a proper, nontrivial, non self-normalizing subgroup of $K$, then $X \times Y$ will be a proper, nontrivial, non self-normalizing subgroup of $G$ so long as either $(X,Y)$ is not $(1_H,1_K)$ or $(H,K)$.  This gives the inequality.

Suppose that $K$ is not nilpotent.  Then we can find 
a subgroup $Y$ of $K$ that is proper and self-normalizing.  Now, $1_H \times Y$ is nontrivial and non self-normalizing in $H\times K$, but is not counted in the previous paragraph. Hence, equality will not hold.  Also, assume that $H$ and $K$ are nilpotent but there exists a prime $p$ which divides both $|H|$ and $|K|$; then we can take $h\in\zent H$ and $k\in\zent K$ both having order $p$. The subgroup $\langle hk\rangle$ of $H\times K$ is nontrivial and non self-normalizing, but it is not counted in the previous paragraph, so again equality does not hold. On the other hand, let $H$ and $K$ be nilpotent of coprime orders, and let $\pi$ be the set of prime divisors of $|H|$. For every subgroup $U$ of $H\times K$ we have $U=\oh{\pi}U\times\oh{\pi'}U$, where $\oh{\pi}U$ is contained in $H$ and $\oh{\pi'}U$ in $K$. Therefore $U$ is of the form considered in the count in the previous paragraph and it is easily seen that equality holds. 
\end{proof}

We can use the previous two results to treat nilpotent groups in general. Note, if $G = H_1 \times \cdots \times H_k$, then, by Lemma \ref{product} and induction, ${\mathcal D}(G) \geq ({\mathcal D}(H_1) +2) \cdots ({\mathcal D}(H_k) + 2) - 2$ and equality will hold when the $H_i$'s are nilpotent and the various $H_i$'s have pairwise coprime orders.   Just as in the $p$-group case, we are obtaining a bound on the nilpotency class and derived length in terms of the number of conjugacy classes of proper nontrivial subgroups.

\begin{theorem}
Suppose $G$ is a noncyclic nilpotent group with order divisible by exactly $k$ distinct primes. Suppose $G$ lies in ${\mathcal D}_n$. Then $G$ has nilpotency class at most $(n+2-2^k)/2^k$. It follows that the derived length of $G$ is at most ${\rm log}_2((n+2-2^k)/2^k) + 1$.
\end{theorem}

\begin{proof} 
Write $G$ as the direct product of its Sylow subgroups, $G = P_1 \times \cdots \times P_k$ (say, $P_i \in {\rm{Syl}}_{p_i} (G)$). By Lemma \ref{product}, it follows
that  
$$n={\mathcal D}(G) \geq ({\mathcal D}(P_1) +2) \cdots ({\mathcal D}(P_k)+ 2) - 2 \geq ({\mathcal D}(P_i) +2)2^{k-1} -2$$ 
for each $1 \leq i \leq k$. Thus, $\displaystyle {\mathcal D} (P_i) + 2 \leq \frac {n +2}{2^{k-1}},$ for each $i$, and so, 
$$\displaystyle {\mathcal D} (P_i) \leq \frac {n +2}{2^{k-1}} - 2 = \frac {n + 2}{2^{k-1}} - \frac {2 \cdot 2^{k-1}}{2^{k-1}}.$$
This yields $\displaystyle {\mathcal D}(P_i) \leq \frac{n+2 - 2^{k}}{2^{k-1}}$. Observe that we have $2\leq {\mathcal{D}}(P_i)$ for at least one $i\in\{1,\dots,k\}$; in fact, assuming the contrary, it is easily seen that all the $P_i$ are cyclic (of order $p_i$ or $p_i^2$) so $G$ is cyclic, not our case.  As a consequence, we get  $2\leq (n+2 - 2^{k})/2^{k-1}$. Now, if ${\mathcal{D}}(P_j)\leq 1$, then the nilpotency class of ${\mathcal{D}}(P_j)$ is $1\leq (n+2 - 2^{k})/2^{k}$; but, by Theorem~\ref{nilpotent}, the same inequality holds if ${\mathcal{D}}(P_j)>1$ as well. The last claim of the statement also follows from \cite[Corollary 4.13]{I}.
\end{proof}

\section{Frobenius Groups} \label{frob sect}

In this section, we count the number of non self-normalizing subgroups for certain families of Frobenius groups. We begin by examining Frobenius groups where both the Frobenius kernel and the Frobenius complement have orders that are prime powers. Recall that if the cyclic Frobenius kernel has order $p^n$ and the cyclic Frobenius complement has odd order $q^m$, where $p$ and $q$ are distinct primes, then $q^m$ must divide $p - 1$.

\begin{lemma} \label{Frob 1}
Let $G$ be a Frobenius group with a cyclic kernel $N$ of order $p^n$ and a cyclic complement $C$ of order $q^m$, where $p$ and $q$ are prime numbers. Then $G \in \mathcal{D}_{(n+1)m - 1}$.
\end{lemma}

\begin{proof}


Let $M$ be a subgroup of $N$ and let $B$ be a proper subgroup of $C$. Since $M$ is normal in $G$, we have that $MB$ and $MC$ are subgroups of $G$, with $MB$ being normal and proper in $MC$. Thus, $MB$ is non self-normalizing. However, since $C$ is self-normalizing, we conclude that $MC$ is self-normalizing via the Correspondence Theorem (again Theorem 3.7 of \cite{I1}).  (In fact, $G/M$ will be a Frobenius group and $MC/M$ will be its Frobenius complement, and thus, $MC$ is self-normalizing.)

On the other hand, suppose $H$ is a subgroup of $G$. Then, $H \cap N$ is normal in $G$ and $G/(H \cap N)$ is Frobenius group.  We see that $H/(H \cap N)$ is conjugate to some subgroup of $C (H \cap N)/(H \cap N)$.  It follows that we can express $H$ as $H = (H \cap N)X$, where $X$ is conjugate to a subgroup of $C$. Furthermore, we see that $H$ will be self-normalizing precisely when $X$ is conjugate to $C$, and thus, $H$ is non self-normalizing whenever $X$ is conjugate to a proper subgroup of $C$.

Next, note that $N$ has $n +1$ subgroups including itself and the trivial subgroup and $C$ has $m$ proper subgroups.  Since $MB$ is nontrivial exactly when $M > 1$ or $B > 1$, we conclude that there are $(n+1)m - 1$ conjugacy classes of non self-normalizing subgroups in $G$.
\end{proof}

We now consider Frobenius groups where the Frobenius kernel is an elementary abelian $p$-group of order $p^2$ for some prime $p$ and a Frobenius complement has prime order $q$.

\begin{lemma} \label{Frob 2}
Let $G$ be a Frobenius group whose kernel $N$ is an elementary abelian $p$-group of order $p^2$ where $p$ is a prime number, and a complement $C$ that has prime order $q$. Then one of the following conditions holds:
\begin{enumerate}
\item If $q$ does not divide $p-1$, then $G \in {\mathcal D}_{\frac{p+1}{q}+1}$. 
\item If $q$ divides $p-1$, then either $G\in {\mathcal D}_{p+2}$ (which always happens for $q=2$) or $G \in {\mathcal D}_{\frac{p-1}{q}+3}$.
\end{enumerate}
\end{lemma}

\begin{proof}
We claim that all the non self-normalizing subgroups of $G$ lie in $N$.  Since $C$ is a Frobenius complement, we see that $C = \norm G C$, and observe that $C$ is a Sylow $q$-subgroup of $G$.  If $H$ is a subgroup of $G$ not contained in $N$, then $q$ divides $|H|$, and so, some conjugate of $C$ lies in $H$.  Without loss of generality, $C \le H$ and so, $\norm G C \le H$.  We have seen that this implies that $\norm G H = H$.  Thus, $N$ contains every non self-normalizing subgroup of $G$.
	
Thus, the only nontrivial non self-normalizing subgroups of $G$ are the subgroups of order $p$, whose count is $p+1$, along with $N$. It follows that $G \in \mathcal{D}_{m+1}$, where $m$ is the count of orbits of the action of $C$ on the set of subgroups of order $p$.

First, suppose that $q$ does not divide $p-1$; then $N$ is a minimal normal subgroup of $G$ and all of the subgroups of order $p$ lie in orbits of size $q$. Hence, we have $m = \frac {p+1}{q}$, and we obtain conclusion (1). 

On the other hand, suppose that $q$ divides $p-1$. In this case, $N$ is not a minimal normal subgroup of $G$ and, by Maschke's Theorem (Theorem 15.2 of \cite{I1}), we see that $N = N_1 \times N_2$, where the direct factors are suitable normal subgroups of $G$ having order $p$. Regarding $N_1$ and $N_2$ as (irreducible) ${\rm {GF}}(p)[C]$-modules, it is not difficult to see that every nontrivial subgroup of $N$ is normal in $G$ if and only if $N_1$ and $N_2$ are isomorphic as modules (which always happens if $q=2$).  In this case, all $p+1$ subgroups of order $p$ are normal in $G$ and those groups along with $N$ give us $p+2$ conjugacy classes of non self-normalizing subgroups and we have that $G \in {\mathcal D}_{p+2}$.  Note that when $p=2$, we know that the Frobenius action of $C \cong Z_2$ on $Z_p$ is inverting every element, so there is only one irreducible module for $Z_2$. In the final case, where $q$ divides $p-1$ and $N_1$ and $N_2$ are not isomorphic as modules, then only $N_1$ and $N_2$ are normal in $G$ among the subgroups of order $p$.  We see that the remaining $p-1$ subgroups of order $p$ lie in orbits of size $q$.  Recalling that we also have $N$, we see that we have $\displaystyle \frac {p-1}q + 3$ conjugacy classes of self-normalizing subgroups, and so, $G \in {\mathcal D}_{\frac{p-1}{q}+3}$.  Therefore, we obtain conclusion (2). 
\end{proof}

Next, we consider a similar situation but assuming that the Frobenius kernel has order $p^3$. 

\begin{lemma} \label{Frob 3}
Let $G$ be a Frobenius group whose kernel $N$ is an elementary abelian $p$-group of order $p^3$ where $p$ is a prime number, and a complement $C$ has prime order $q$. Then one of the following conditions holds:
\begin{enumerate}
\item If $q$ does not divide $p-1$, then $G\in {\mathcal D}_n$ where $\displaystyle n = {\frac{2(p^2+p+1)}{q}+1}$. 
\item If $q$ divides $p-1$, then $G\in {\mathcal D}_n$ where $n$ is one of $2p^2+2p+3$ (which always happens for $q=2$), $\displaystyle \frac{2(p^2-1)}{q}+2p+5$, or $\displaystyle \frac{2(p^2+p-2)}{q}+7$.
\end{enumerate}
\end{lemma}

\begin{proof}
As in the previous lemma, the nontrivial non self-normalizing subgroups of $G$ are those of order $p$ or $p^2$, together with $N$. Therefore, we have to count the number of orbits of the action of $C$ on the subgroups of order $p^i$, with $i=1, 2$. Note that $N$ has $p^2 + p + 1$ subgroups of order $p$ and of order $p^2$. 

If $q$ does not divide $p-1$, then $N$ is a minimal normal subgroup of $G$; therefore, every subgroup of order $p$ or $p^2$ lies in an orbit of size $q$ for the action of $C$, and we get (1).

Assume now that $q$ divides $p-1$. Then the ${\rm {GF}}(p)[C]$-module $N$ has only $1$-dimensional irreducible constituents, and the relevant number of orbits depends on the dimensions of the homogeneous components of this module. In particular, the first value of $n$ displayed in (2) occurs when $N$ is homogeneous; in this case every subgroup of $N$ is normal in $G$, so we have $p^2+p+1$ orbits of size $1$ in the action of $C$ on the set of subgroups of order $p$ of $G$, and the same holds for the subgroups of order $p^2$. Taking into account the contribution of $N$ we then get $n=2(p^2+p+1)+1=2p^2+2p+3$. The second value in (2) occurs when $N$ has two homogeneous components $T$ and $U$, having dimensions $1$ and $2$ respectively; in this case every subgroup of order $p$ of $U$ is normal in $G$, and of course also $T$ is a normal subgroup of $G$ of order $p$, so we have a total of $p+2$ orbits of size $1$ in the action of $C$ on the set of subgroups of order $p$ of $G$. On the other hand, the remaining $p^2+p+1-(p+2)=p^2-1$ subgroups of order $p$ lie in orbits of size $q$, thus we get $(p^2-1)/q$ orbits of this kind; taking into account that the same holds for the subgroups of order $p^2$, and counting also the contribution of $N$, we get $n=\frac{2(p^2-1)}{q}+2p+5$. Finally, the third value of $n$ occurs when $N$ has three ($1$-dimensional) homogeneous components; in this case, among the subgroups of order $p$, only the three homogeneous components are normal in $G$, therefore we get three orbits of size $1$ and $(p^2+p-2)/q$ orbits of size $q$ in the action of $C$ on the set of subgroups of order $p$ of $G$. The same holds for the subgroups of order $p^2$ and, taking into account $N$, we get $n=\frac{2(p^2+p-2)}{q}+7$.
\end{proof}

\section{Center}

In this section we focus on the relationship between ${\mathcal D}(G)$ and ${\mathcal D}(\zent G)$. 

\begin{proposition} \label{center} Suppose $G$ is a nonabelian group with a nontrivial center. Then we have $${\mathcal{D}}(G)-{\mathcal{D}}(\zent G)\geq 3.$$
\end{proposition}
\begin{proof} Set $Z=\zent G$ and, for a proof by contradiction, assume ${\mathcal{D}}(G)-{\mathcal{D}}(Z)\leq 2$. By Lemma~\ref{rem} we get ${\mathcal{D}}(G/Z)\leq 1$, but, since $G/Z$ cannot be cyclic, by propositions~\ref{D_0} and \ref{D_1} the only possibility is that $G/Z$ is a Frobenius group with a kernel $N/Z$ of prime order $p$ and complements of prime order~$q$. Observe that $N$ is abelian and, together with $Z$, it gives two conjugacy classes of nontrivial, proper, non self-normalizing subgroups of $G$ which add to ${\mathcal{D}}(G)-{\mathcal{D}}(Z)$.

Now, if $N$ is not a $p$-group, then its Sylow $p$-subgroup is normal in $G$ and not contained in $Z$, yielding a third conjugacy class of subgroups of $G$ that leads to a contradiction. On the other hand, let $N$ be a $p$-group: then, denoting by $Q$ a Sylow $q$-subgroup of $G$, we have $G=NQ$. By coprimality we get $N=[N,Q]\times\cent N Q$, where clearly $\cent N Q=Z$, thus $G=[N,Q]Q\times Z$. As a consequence, $[N,Q]Q$ is another nontrivial, proper, non self-normalizing subgroup of $G$ that adds to ${\mathcal{D}}(G)-{\mathcal{D}}(Z)$, yielding the final contradiction. 
\end{proof}

It follows that groups in ${\mathcal D}_k$ for $0 \leq k \leq 2$ are either abelian or have trivial centers (as already seen in Section~\ref{sect:prelim}).
The following example shows that this bound for $k$ is sharp.

\begin{example} \label{example} 
For primes $p$ and $q$ such that $q$ divides $p-1$, let $K$ be a group of order $p$ and $H$ a cyclic group of order $q^2$, and consider the group $G = K\rtimes H$ where $\cent H K$ has order $q$.

Observe that $Z=\zent G=\cent H K$ has order $q$, so ${\mathcal D}(Z) = 0$.
Now, the subgroups $Z$, $K$ and $KZ$ give a contribution of $3$ to ${\mathcal D}(G)$, and it is clear that there are no other conjugacy classes of nontrivial, proper, non self-normalizing subgroups in $G$. Hence, ${\mathcal D}(G)-{\mathcal D}(Z) = 3-0=3$.
\end{example}	

It turns out that if, for a nonabelian group $G$ with nontrivial center, the bound in Proposition \ref{center} is attained, then $G$ is in fact as in the previous example. We will show this in the remaining part of this section (see Corollary~\ref{D(G)-D(Z)=3}), taking into account that ${\mathcal D}(G) - {\mathcal D}(\zent G) =3$ yields ${\mathcal D}(G/\zent G) \leq 2$  by Lemma~\ref{rem} (but, as already observed, ${\mathcal D}(G/\zent G)$ cannot be $0$).

We consider first the case where ${\mathcal D} (G/\zent G) = 1$.

\begin{lemma} \label{D(G/Z) = 1}
Suppose ${\mathcal D}(G)-{\mathcal D}(\zent G) = 3$ and ${\mathcal D}(G/\zent G) =1$. Then the structure of $G$ is as in Example~$\ref{example}$.
\end{lemma}

\begin{proof} Set $Z=\zent G$. As in the first paragraph in the proof of Proposition~\ref{center}, we see that $G/Z$ is a Frobenius group with a kernel $N/Z$ of prime order $p$ and complements of prime order~$q$ (in particular, $q$ divides $p-1$). Clearly, as
$N/Z$ is cyclic and $Z$ is central in $N$, the subgroup $N$ is abelian. Setting $M=\oh{q'} N$ we have that $M$ is normal in $G$ and, if $Q$ is a Sylow $q$-subgroup of $G$, then $G$ is a semidirect product $MQ$.  By coprimality, we see
that $M = [M,Q] \times \cent M Q$ where, in fact, $\cent M Q = M \cap Z$; therefore,
$|[M,Q]| = |M:M \cap Z| = |N:Z| = p$.

Observe that $G = M Q = ([M,Q] \times (M \cap Z)) Q = [M,Q]Q \times (M \cap
Z)$.  Thus, $[M,Q]$ and $[M,Q]Q$ are normal in $G$.  Also, $N$ and $Z$ are normal
in $G$.  If $M \cap Z > 1$, then these four subgroups of $G$ are distinct, proper, nontrivial  and not properly contained in $Z$, which contradicts the
hypothesis.  Thus, we must have $M \cap Z = 1$.

Now, $M = [M,Q]$ has order $p$; moreover, $Z$ is a $q$-group, $N=M\times Z$, and the pairwise distinct normal subgroups $M$, $Z$, $N$ all contribute with $1$ to the difference ${\mathcal D}(G)-{\mathcal D}(Z)$.
If $|Z| > q$ then $Z$ has a subgroup $X$ of order $q$, and $MX$ will
be a fourth normal, nontrivial, proper subgroup of $G$ that is not properly
contained in $Z$, contradicting the hypothesis.  Thus, we have
$|Z| = q$ and $|Q| = q^2$.

If $Q$ is not cyclic, then there exists a subgroup
$Y$ of order $q$ in $Q$ so that $Y \ne Z$, and $MY$ is a fourth normal,
proper, nontrivial subgroup of $G$ that is not properly contained in $Z$.
This again contradicts the hypothesis, and the desired conclusion follows.
\end{proof}

As for the case ${\mathcal D} (G/\zent G) = 2$, we prove next that this is not compatible with the condition ${\mathcal D}(G)-{\mathcal D}(\zent G) = 3$.

\begin{lemma} \label {D (G/Z) = 2}
If ${\mathcal D} (G/\zent G) = 2$ and $\zent G> 1$, then ${\mathcal D}(G) - {\mathcal D}(\zent G)\geq 4$.
\end{lemma}

\begin{proof}
Since ${\mathcal D} (G/\zent G) = 2$, we know that $G$ is not abelian and, in particular, $G/\zent G$ is not cyclic. Therefore, as will be proved in Proposition~\ref{D_2},  $G/\zent G$ is a Frobenius group (either of the form $C_{p^2} \rtimes C_q$ or isomorphic to ${\rm A}_4$; in the latter case, we set $p=2$ and $q=3$).   
Let $N/\zent G$ be the Frobenius kernel of $G/\zent G$.  

Let $P$ be a Sylow $p$-subgroup of $N$ and let $X$ be a subgroup of index $p$ in $P$.  If $\zent G$ is not a $p$-group, then we consider the (nontrivial) $p$-complement $Y$ of $Z$. Clearly $P$ is not self-normalizing in $G$, as $\norm G P$ contains $Y$; moreover, $X$ is normal in $P$ (hence, not self-normalizing in $G$) and not contained in $\zent G$.  Note that $XY$ is also a subgroup of $G$ not contained in $\zent G$ that is normalized by $P$.  Thus we see that $N$, $P$, $X$, and $XY$ represent four conjugacy classes of proper non self-normalizing subgroups of $G$ that are not contained in $\zent G$, and our claim is proved in this case.

Finally, assume that $\zent G$ is a $p$-group. If $Q$ is a Sylow $q$-subgroup of $G$, then $Q$ is normal in $Q \zent G$, so $Q$ is not self-normalizing in $G$.  We conclude that $N$, $X$, $\zent G$ and $Q$ all are non self-normalizing subgroups that are not properly contained in $\zent G$, which yields the result in this case as well. 
\end{proof}

As an immediate consequence of the above results, we get the following.

\begin{corollary} \label{D(G)-D(Z)=3}
Suppose $G$ is a nonabelian group with a nontrivial center. If $G$ lies in ${\mathcal D}_n$, then one of the following conclusions hold.
\begin{enumerate}
\item $\zent G$ lies in ${\mathcal D}_{n-k}$ for $k \geq 4$;
\item $G=KH$ where $K\trianglelefteq G$ has prime order $p$, $H$ is a cyclic subgroup of order $q^2$ for a suitable prime $q$ dividing $p-1$, and $\cent H K=\zent G$ has order $q$. In this case we have $G\in{\mathcal D}_3$ and $\zent G\in{\mathcal D}_0$.
\end{enumerate}
\end{corollary}


\section{Classification:  ${\mathcal D}_2$ and ${\mathcal D_3}$} \label{sect:D_3}


In this section, we classify ${\mathcal D}_2$ and ${\mathcal D_3}$.  Next, we consider ${\mathcal D}_2$.  Notice that we obtain examples that are nonabelian and solvable with derived length equal to $2$.

\begin{proposition} \label{D_2}
We have $G \in {\mathcal D}_2$ if and only if $G$ is one of the following:
\begin{enumerate}
	\item $C_{p^3}$ for some prime $p$.
	\item $C_{pq}$ for distinct primes $p, q$.
	\item A Frobenius group of the form $C_{p^2} \rtimes C_q$. 
	\item ${\rm A}_4$.  
\end{enumerate}
\end{proposition}

\begin{proof}
Clearly, if $G$ is $C_{p^3}$ or $C_{pq}$ for $p$,$q$ distinct primes, then $G\in{\mathcal{D}}_2$; also, if $G$ is a Frobenius group of the form $C_{p^2} \rtimes C_q$ where $q$ divides $p-1$, or ${\rm A}_4$, then $G \in {\mathcal D}_2$ by Lemma \ref{Frob 1} and Lemma \ref{Frob 2}(1) respectively.  
	
Conversely, suppose that $G \in {\mathcal D}_2$.  By Proposition \ref{metabelian}, we know $G$ is solvable and its derived length is at most $2$.  If its derived length is $1$, then $G$ is abelian, and all subgroups are normal in $G$.  Since $G \in {\mathcal D}_2$, we see that $G$ has only two proper, nontrivial subgroups, and it is not difficult to determine that this implies that either $G$ is cyclic of order $p^3$ for some prime $p$ or $G = C_{pq}$ for distinct primes $p$ and $q$.  
	
In the remaining cases, $G$ has derived length $2$, so that $G'$ will be a proper, nontrivial, abelian, non self-normalizing subgroup of $G$.  By Lemma \ref{rem}, $G/G'$ has at most one proper nontrivial subgroup, and we first assume that it has precisely one such subgroup (hence, $G/G'$ is a  cyclic group of order $q^2$). Under this assumption, the two proper nontrivial non self-normalizing subgroups of $G$ are then $G'$ and a normal subgroup of $G$ of index $q$. As a consequence, $G'$ must be a group of prime order.  We have that $G$ is not a $q$-group: otherwise $G'$ would be central and, since $G / \zent G$ is cyclic, $G$ would be abelian. Therefore, $G$ is isomorphic to a group of the form $C_p \rtimes C_q^2$; but in this case we get three nontrivial, non self-normalizing subgroups of order $p$, $q$ and $pq$. We conclude that $G/G'$ does not have any proper nontrivial subgroup, so $G/G' \cong C_q$ for some prime $q$. 
	
Assume now that $G'$ is cyclic; then it  has order coprime to $q$ (as otherwise $G$ would have a subgroup of index $q^2$) and  its subgroups are characteristic and so normal in $G$.  This implies that $G'$ has at most one proper, nontrivial subgroup.  Note that if $G'$ has no proper nontrivial subgroups, then $G'$ has order $p$, and since $G$ is not abelian, this implies that $G$ is a Frobenius group of order $pq$, and by Lemma \ref{Frob 1}, we have $G \in {\mathcal{D}}_1$ which is a contradiction. 

Thus $G'$ has one proper, nontrivial subgroup, and Proposition~\ref{D_1} yields that $G'$ is cyclic of order $p^2$.  Note that $G'$ and its subgroup of order $p$ account for all of the nontrivial non self-normalizing subgroups of $G$, so a Sylow $q$-subgroup of $G$ must be self-normalizing and $G$ is hence a Frobenius group as in (3).
	
Finally, suppose that $G'$ is not cyclic.  Since $G'$ is abelian, all its  proper nontrivial subgroups are not self-normalizing in $G$. It follows that they must be conjugate in $G$, and therefore $G'$ must be an elementary abelian $p$-group for a suitable prime $p$; moreover, $G'$ has order $p^2$. Since there are $p + 1$ subgroups of order $p$, we have $p + 1 = q$. Thus, $p = 2$, $q = 3$ and $G\cong {\rm A}_4$.
\end{proof}

We next classify groups with three conjugacy classes of nontrivial, non self-normalizing subgroups.  However, beyond this, we think the number of isomorphism classes of groups is going to become quite large.  As we proved in Proposition \ref{metabelian}, we see that the nonabelian examples in this case are solvable with derived length equal to $2$.

\begin{proposition} \label{D_3}
We have $G \in {\mathcal D}_3$ if and only if $G$ is one of the following:
\begin{enumerate}
	\item $C_{p^4}$ for some prime $p$.
	\item $C_2 \times C_2$.
	\item A Frobenius group of the form $C_{p^3} \rtimes C_q$, for $p$ and $q$ primes.
	\item A Frobenius group of the form $(C_p \times C_p) \rtimes C_q$, for $p$ and $q\neq 2$ primes such that $p=2q-1$.
	\item A Frobenius group of the form $C_p \rtimes C_{q^2}$, for $p$ and $q$ primes. 
	\item A Frobenius group of the form $C_{pr} \rtimes C_q$, for $p$, $q$ and $r$ primes. 
	\item A Frobenius group of the form $(C_p \times C_p \times C_p) \rtimes C_q$, for $p$ and $q$ primes such that $q = p^2+p+1$.
	\item A group of the form $C_p \rtimes C_{q^2}$ with $|\zent G|=q$, for $p$ and $q$ primes. 
\end{enumerate}
\end{proposition}

\begin{proof}
Since $C_{p^4}$ and $C_2 \times C_2$ have three proper, nontrivial subgroups which are normal, these abelian groups lie in ${\mathcal D}_3$.  We use Lemma \ref{Frob 1} to see that the groups as in (3) and in (5) lie in ${\mathcal D}_3$.  Applying Lemma \ref{Frob 2} and Lemma \ref{Frob 3}, we see that the groups in (4) and (7) (respectively) are in ${\mathcal D}_3$ as well.  Now, let $G$ be a group as in (6); then it is easy to see that $G$ has three normal subgroups, of orders $p$, $q$, $pr$ respectively, whereas any other nontrivial subgroup is self-normalizing. Therefore such a group $G$ lies in ${\mathcal D}_3$.
Finally, as shown in Example~\ref{example}, a group as in (8) lies in ${\mathcal D}_3$.

Conversely, suppose that $G$ lies in ${\mathcal D}_3$.  By Proposition \ref{metabelian} we know that $G$ is solvable with derived length at most $2$.  If $G$ is abelian, then $G$ has three proper, nontrivial subgroups.  If $G$ is cyclic, then $G \cong C_{p^4}$ for some prime $p$.  If $G$ is not cyclic, then it has a subgroup isomorphic to $C_p \times C_p$, and $C_p \times C_p$ has $p+1$ subgroups of order $p$.  Thus, in this case the only possibility is $G\cong C_2 \times C_2$.

We now suppose that $G$ is not abelian. Under this assumption, if $\zent G > 1$, then Corollary \ref{D(G)-D(Z)=3} yields that $G$ is as in (8).  Thus, it remains to consider the case when $\zent G = 1$.  

We know that $G$ has derived length equal to $2$, so $G'$ is a proper, nontrivial abelian subgroup.  By Lemma \ref{rem}, we know that ${\mathcal D}(G/G') \le 2$.  We first suppose that $G/G'$ has no proper, nontrivial non self-normalizing subgroups.  Using Proposition~\ref{D_0}, we see that $G/G' \cong C_q$ for some prime $q$. 

If $G'$ is cyclic, then its subgroups are characteristic and so normal in $G$;  this implies that $G'$ has at most two proper, nontrivial subgroups.  Note that if $G'$ has no or one proper nontrivial subgroups, then $G'$ has order $p$ or $p^2$ where $p$ is a prime; moreover, since we are assuming $\zent G= 1$, we have that $p\neq q$, and no nontrivial element of $G'$ can be centralized by an element of order $q$. It follows that $G$ is a Frobenius group of order $pq$ or $p^2q$; by Lemma \ref{Frob 1} we have $G \in {\mathcal{D}}_1$ or ${\mathcal D}_2$, which is a contradiction.  Thus $G'$ has two proper, nontrivial subgroups, which implies that $G'$ has order $p^3$ for a prime $p$, or $G'$ has order $pr$ where $p$ and $r$ are distinct primes. As above, we easily see that $q$ is coprime with $|G'|$, and $G$ is a Frobenius group 
as in (3) or (6) in this case.

Still assuming $G/G'\cong C_q$, we now suppose that $G'$ is not cyclic: in this case, there exists a prime $p$ such that $G'$ has a subgroup isomorphic to $C_p\times C_p$. Our first claim is that $G'$ is a $p$-group (again, with $p\neq q$). In fact, if $|G'|$ is divisible by a prime $r\neq p$, then the Sylow $p$-subgroup and the Sylow $r$-subgroup of $G'$ are nontrivial, proper subgroups of $G'$ that are normal in $G$; these, together with a (proper) subgroup of order $pr$ of $G'$ and $G'$ itself, yield a contradiction. Next, $G'$ is necessarily an elementary abelian $p$-group, as otherwise it would have proper subgroups isomorphic to $C_p$, $C_{p^2}$ and $C_p\times C_p$ which, together with $G'$, again yield a contradiction. Arguing similarly, it is not difficult to see that $G'$ must be isomorphic either to $C_p\times C_p$ or to $C_p\times C_p\times C_p$ and (recalling that $\zent G=1$) that $G$ is a Frobenius group. Using now Lemma~\ref{Frob 2} and Lemma~\ref{Frob 3}, we get that $G$ is as in (4) or (7) respectively. 



Next, we suppose that $G/G'$ has one proper, nontrivial non self-normalizing subgroup $N/G'$. Since $G/G'$ is abelian, we apply Proposition~\ref{D_1} to see that $G/G' \cong C_{q^2}$ for some prime $q$. Let $Q$ be a Sylow $q$-subgroup of $G$ (note that $|Q|$ is at least $q^2$). Considering a subgroup $T$ of $Q$ having order $q$, we see that the $G$-conjugacy classes of $T$, $G'$ and $N$ account for all the $G$-conjugacy classes of nontrivial, proper, non self-normalizing subgroups of $G$. Note that, as a consequence, $|G'|$ must be a power of a prime $p$, with $p\neq q$ again because $\zent G=1$; hence $T$ is not contained in $G'$. Now it is immediate to see that $G'$ lies in ${\mathcal D}_0$, thus $G'\cong C_p$. To sum up, $G$ is a group of the form $C_p\rtimes C_{q^2}$ with no element of order $pq$ (as otherwise $\zent G$ would contain an element of order $q$), therefore $G$ is a Frobenius group as in (5).


The remaining case is that $G/G'$ has two proper, nontrivial non self-normalizing subgroups (which implies $G'\in{\mathcal D}_0$).   Since $G/G'$ is abelian, we apply Proposition \ref{D_2} to see that $G/G' \cong C_{q^3}$ for some prime $q$ or $G/G' \cong C_{qr}$ for distinct primes $q$ and $r$. In the former case, let $N$ and $U$ be subgroups of $G$, both containing $G'$, having index $q$ and $q^2$ in $G$ respectively; then $G'$, $U$ and $N$ account for all the $G$-conjugacy classes of nontrivial, proper, non self-normalizing subgroups of $G$, but now we get a contradiction by considering a subgroup of order $q$ of $G$ (which is certainly different from $G'$ as otherwise $\zent G$ would be nontrivial).



Hence, we have $|G:G'| = qr$.  Recalling that $G'$ lies in ${\mathcal D}_0$, we have that $|G'|$ is a prime $p$ that we can assume is different from $r$. Let $N$ and $M$ be subgroups of $G$, both containing $G'$, so that $|G:N| = q$ and $|G:M| = r$. Now $G'$, $M$ and $N$ account for all the $G$-conjugacy classes of nontrivial, proper, non self-normalizing subgroups of $G$. If $p\neq q$, then $G\cong C_p\rtimes C_{qr}$ has also a $G$-conjugacy class of non self-normalizing subgroups of order $q$, yielding a contradiction. On the other hand, if $p=q$ then, denoting by $R$ a Sylow $r$-subgroup of $G$, we have $M=[M,R]\times \cent M R$; but the fact that $\zent G=1$ yields $\cent M R=1$, hence $M=[M,R]\subseteq G'$, the final contradiction.
\end{proof}

\section{Solvability}

Using Proposition \ref{metabelian} we can prove that, if a solvable group lies in $\mathcal{D}_n$ for $n\geq 3$, then its derived length is at most $n-1$. 

\begin{proposition}\label{derivedlength}
Let $G$ be a solvable group. If $G\in\mathcal{D}_n$ with $n\geq 3$, then the derived length $d$ of $G$ is at most $n-1$.
\end{proposition}

\begin{proof}
Let $G$ be a minimal counterexample to the statement, and let $N$ be the abelian subgroup $G^{(d-1)}$. Since $N$ is clearly nontrivial, we are in a position to apply Lemma~\ref{rem} obtaining that $G/N$ lies in $\mathcal{D}_m$ for a suitable $m\leq n-1$. If $m\geq 3$, our minimality assumption yields that the derived length of $G/N$ is at most $m-1\leq n-2$, thus $d\leq n-2+1=n-1$, contradicting the fact that $G$ is a counterexample. On the other hand, if $m\leq 2$, then the derived length of $G/N$ is at most $2$ by Proposition~\ref{metabelian}, and hence $d\leq 3$; this forces $n\leq 3$ (otherwise $G$ is not a counterexample), but then again Proposition~\ref{metabelian} yields the final contradiction, and we conclude that $d\leq 2$. 
\end{proof}

It can be easily checked that the alternating group $\textnormal{A}_5$ lies in $\mathcal{D}_4$, therefore the conclusion of Proposition~\ref{metabelian} does not hold already for $n=4$. In fact, $\textnormal{A}_5$ turns out to be the unique nonsolvable group belonging to $\mathcal{D}_4$.  We prove that now.

\begin{theorem}\label{NonsolvableD4} 
Let $G$ be a nonsolvable group. If $G\in\mathcal{D}_4$, then $G\cong\textnormal{A}_5$.
\end{theorem}

\begin{proof}
Assume that $G$ is a nonsolvable group in ${\mathcal D}_4$.  Suppose first that $N$ is a nontrivial normal subgroup of $G$: then, in view of Lemma~\ref{rem}, $G/N$ lies in $\mathcal{D}_n$ for a suitable $n \leq 3$, therefore Proposition~\ref{metabelian} yields that $G/N$ is solvable. We deduce that, in particular, the solvable radical of $G$ is trivial and every minimal normal subgroup of $G$ is nonabelian.

Next, we claim that $G$ has a unique minimal normal subgroup, which is simple. In fact, if $M_1$ and $M_2$ are distinct minimal normal subgroups of $G$, then both of them have a subgroup of order $2$ that is not self-normalizing (and these two subgroups are clearly not conjugate in $G$); as $M_1$ and $M_2$ are obviously not self-normalizing and not conjugate in $G$, we conclude that $G=M_1\times M_2$, which is against the previous paragraph because $G/M_1\cong M_2$ is not solvable. Now, let $M$ be the unique minimal normal subgroup of $G$: we have that $M=S_1\times\cdots\times S_k$ is the direct product of $k$ isomorphic nonabelian simple groups for a suitable positive integer $k$. If $k>1$ then every subgroup of $S_1$ is not self-normalizing in $G$: this holds in particular for a subgroup of order $2$ and for every Sylow subgroup of $S_1$, which easily yields a contradiction taking into account that also $M$ is not self-normalizing. 

So, $G$ is an almost-simple group, and we assume that $G$ is not simple. Denoting by $S$ the socle of $G$, let $p$ and $q$ be odd prime divisors of $|S|$, and let $D$, $P$, $Q$ be Sylow subgroups of $S$ for the primes $2$, $p$, $q$ respectively. Now, any subgroup of order $2$ of $G$ is not self-normalizing in $G$, as well as $S$; if $|P|\geq p^3$, then we have subgroups of order $p$ and $p^2$ that are not self-normalizing in $G$ as well, and this implies $|Q|=q$. In particular, $Q$ is abelian and self-normalizing, thus $G$ has a normal $q$-complement and we get a contradiction. Similarly we see that $|Q|\leq q^2$, hence both $P$ and $Q$ are abelian. If $|D|\geq 2^3$ then, together with a subgroup of order $2$ and $S$, we also have a subgroup of order $4$ that is not self-normalizing in $G$; as a consequence, one among $P$ and $Q$ must be self-normalizing. Assuming this is $P$ (without loss of generality), then $G$ has a normal $p$-complement, and now also $Q$ must be self-normalizing. This yields the final contradiction that $G$ has also a normal $q$-complement, and our conclusion so far is that $G$ is a simple group.  

We claim that the order of $G$ is divisible by precisely three primes. In fact, assume that $|G|$ is divisible by three distinct odd primes $p$, $q$, $r$, and let $P\in\syl p G$, $Q\in\syl q G$, $R\in\syl r G$. Assume that one among $P$, $Q$ and $R$ is nonabelian: without loss of generality, let this be $P$. Since in this situation we have $|P|\geq p^3$, $G$ has subgroups of order $2$, $p$ and $p^2$ that are not self-normalizing. Therefore we have $|Q|\leq q^2$ and $|R|\leq r^2$, so that both $Q$ and $R$ are abelian. Moreover, at least one among $Q$ and $R$ must be self-normalizing; if this is $Q$, then $G$ has a normal $q$-complement. But now also $R$ must be self-normalizing and $G$ has a normal $r$-complement as well, a contradiction. Our conclusion so far is that $P$, $Q$ and $R$ are all abelian, and we claim that in fact also a Sylow $2$-subgroup $D$ of $G$ is abelian. In fact, assuming the contrary, $G$ has subgroups of order $2$ and $4$ that are not self-normalizing, therefore at least one among $P$, $Q$, $R$ must be self-normalizing; if this is $P$ (without loss of generality), then $G$ has a normal $p$-complement and now one among $Q$ and $R$ must be self-normalizing. Assuming this is $Q$, $G$ has a normal $q$-complement. Finally, $R$ is forced to be self-normalizing yielding the existence of a normal $r$-complement and thus a contradiction. Given that now $D$, $P$, $Q$ and $R$ are all abelian and $G$ has a subgroup of order $2$ that is not self-normalizing, an entirely similar argument yields again a contradiction.

Our claim that $|G|$ is divisible by three primes is then proved. Now, the set of simple groups having this property is finite, and its elements are $${\rm A}_5,\; {\rm A}_6, \;{\rm L}_2(7),\; {\rm L}_2(8),\; {\rm L}_2(17),\; {\rm L}_3(3),\; {\rm U}_3(3),\; {\rm U}_4(2)$$
(see \cite{He}). The last four groups cannot lie in $\mathcal{D}_4$ because of the orders of their Sylow subgroups, whereas a direct check rules out all the others except ${\rm A}_5$: in fact we have ${\rm A}_6\in{\mathcal D}_{11}$, ${\rm L}_2(7)\in{\mathcal D}_{8}$ and ${\rm L}_2(8)\in{\mathcal D}_{6}$. The proof is complete.
\end{proof}

It can be checked (by direct computation or using GAP or Magma) that ${\rm SL}_2 (3)$ lies in  ${\mathcal D}_4$. In fact, as we will show next, ${\rm SL}_2 (3)$ is the unique solvable group in ${\mathcal D}_4$ whose derived length is larger than $2$.

\begin{theorem} \label{D=4,dl} 
Suppose $G$ lies in ${\mathcal D}_4$ and $G \neq {\rm A}_5$. Then $G$ is solvable and either has derived length at most $2$ or is isomorphic to ${\rm SL}_2 (3)$. 
\end{theorem}

\begin{proof} 
By Proposition \ref{NonsolvableD4}, it follows that $G$ is solvable.   We suppose $G$ is not abelian and thus $G'$ is nontrivial.  Then $G/G'$ lies in ${\mathcal D}_m$ for some $m\leq 3$ by Lemma~\ref{rem}.

If $G/G'$ is not cyclic, an application of Propositions~\ref{D_0}, \ref{D_1}, \ref{D_2} and \ref{D_3} yields 
$G/G' \cong C_2 \times C_2$.  Now, $G'$ and the maximal subgroups of index $2$ account for all the conjugacy classes of nontrivial, non self-normalizing subgroups of $G$.  Since $G$ is solvable, this implies that $G'$ must have a prime order, and so $G$ will have derived length $2$.  

Therefore, we suppose $G/G'$ is cyclic. Let $N \unlhd G$ be maximal such that $G/N$ is not abelian. By the comment following Lemma~12.3 in \cite{Isaacs}, we know $G/N$ is a $p$-group for some prime $p$ or a Frobenius group. However,  if $G/N$ is a nonabelian $p$-group, then $G/G'$ cannot be cyclic which is a contradiction.  Thus, $G/N$ is a Frobenius group with abelian Frobenius kernel $G'N/N$, and we can assume $N\neq 1$.  Note $G'$ and $N$ are distinct normal subgroups of $G$ and thus, we only have two more conjugacy classes of nontrivial non self-normalizing subgroups in $G$. 

Next, suppose $N$ is not contained in $G'$; so, $G'N$ is an additional normal subgroup of $G$. Also, note that $G'' \leq G' \cap N$.  Hence, if $G' \cap N = 1$, then the result follows, so we may assume that $1 \neq G'' = G' \cap N$ and that we have four distinct proper nontrivial normal subgroups, namely, $G', G'', N$ and $NG'$.

As $G$ is solvable, this will force various quotients to be cyclic of prime order.  In particular, $|G/G'N| = p$, $|G'N/N| = q$, $|N/G''| = r$ and $|G''| = s$ where $p,q,r$ and $s$ are primes. So $|G'| = qs$. If $q \neq s$ then the Sylow $q$-subgroups of $G'$ give another conjugacy class of non self-normalizing subgroups by the Frattini argument, which is a contradiction. Thus $q=s$ and $|G'| = q^2$. But then $G'$ is abelian and the result follows.

So we assume $N \leq G'$. Let $H/N$ be the Frobenius complement in $G/N$. Suppose $H/N$ is not cyclic of prime order. Then $H/N$ has a proper subgroup of prime order, call this subgroup $P/N$. Then $P$, $G'P$, $G'$ and $N$ represent four conjugacy classes of nontrivial, proper, non self-normalizing subgroups. Note that we can assume $G''=N$ (otherwise $G''=1$ and we are done), and $N$ is abelian. As $G$ lies in ${\mathcal D}_4$, it follows that $N$ and $G'/N$ must be cyclic of prime order. But then the automorphism group of $N$ will be abelian which forces $G'$ to centralize $N$, and thus, $G'$ is central-by-cyclic.  Therefore, $G'$ is abelian, as required. 
  
We may assume $H/N \cong G/G'$ is cyclic of prime order $p$, say.  By Lemma~\ref{rem}, we have that $G/N$ lies in ${\mathcal D}_m$ for some $m \leq 3$. By the choice of $N$ it follows that $G'/N$ is a minimal normal subgroup of the solvable group $G/N$, and so, $G'/N$ is elementary abelian of $q$-power order for some prime $q$.  We will consider next the case when $G'$ is not a $q$-group.

Let $Q$ be a Sylow $q$-subgroup of $G'$ (and hence of $G$). By the Frattini argument, $Q$ is not self-normalizing. This argument holds for all prime divisors of $G'$.  Thus, if $|G'|$ is divisible by three distinct primes,  then consideration of the Sylow subgroups of $G'$ along with the subgroups $G'$ and $N$ gives a contradiction. Thus, $|G'|$ is divisible by two primes, $q$ and $r$ say (possibly, $r=p$).  

Now, suppose $|G'/N|>q$ and consider a subgroup $L$ of $G'$ such that $L\supseteq N$ and $|L/N|=q$. Also, let $Q_0$ be a Sylow $q$-subgroup of $L$. Since $L\trianglelefteq G'$, the Frattini argument yields that $Q_0$ is not a self-normalizing subgroup of $G$, and the subgroups $G'$, $L$, $N$, $Q$, $Q_0$ represent too many conjugacy classes of nontrivial, proper, non self-normalizing subgroups of $G$. Therefore we have $|G'/N|=q$. If $q$ divides $|N|$, then $G'$, $N$, $Q$, $R\in\syl r{G'}$ and $Q\cap N$ produce again the same contradiction. As a consequence, $N$ must be an $r$-group (in fact, the Sylow $r$-subgroup of $G'$) and $|Q|=q$. 

Observe that $|N|$ is clearly at most $r^2$, and assume for the moment $|N| =r$.  Then either $r \neq p$ and all Sylow subgroups of $G$ are cyclic, which implies (via Theorem~\ref{z-group}) that $G$ has derived length at most $2$, or $r=p$.  But in the latter case, as $p$ divides $q-1$, it follows that $N$ is central in $G'$ and thus $G'$ is abelian. In any case we get the desired conclusion.

Hence, we may assume $|N| = r^2$ (i.e., $|G| = pq r^2$), and the four conjugacy classes of nontrivial, proper, non self-normalizing subgroups are represented by $G'$, $N$, $Q$ and $M$ where $M$ is a subgroup of $N$ of order $r$. Now, denoting by $T$ a Hall $\{p,q\}$-subgroup of $G$, we have $G=NT$ and so $T\cong G/N$ is a Frobenius group (note that $p$ has to divide $q-1$, hence $p<q$). We claim that, assuming $G'$ nonabelian (otherwise we are done), also $G$ is a Frobenius group with kernel $N$ and complement $T$. In fact, since a Sylow $p$-subgroup of $G$ is necessarily self-normalizing, we cannot have elements of order $pr$ in $G$. Moreover, if $G$ has an element $x$ of order $qr$, then $|\norm G Q|=pqr$ (recall that $Q$ is not normal in $G$, since otherwise $G'=N\times Q$ would be abelian) and $|\syl q G|=r$; in particular, $q$ divides $r-1$ and $\langle x\rangle$ is normal in $\norm G Q$, which yields a fifth conjugacy class of nontrivial, proper, non self-normalizing subgroups of $G$. This contradiction shows that there are no elements of order $qr$ in $G$ as well, hence $G$ is a Frobenius group as claimed. However, by Theorem~16.12 b) in \cite{Hu2}, the Frobenius complement $T$ of order $pq$ should be cyclic, contradicting the fact that $T$ itself is a Frobenius group; we deduce that our assumptions lead to a contradiction, and we are left with the case $|G|=p^3q$.


Now we have $|G/G'| =p$, $|G'/N| = q$, and $|N| = p^2$; in this situation, the subgroups $N$, $G'$, $Q$ together with a subgroup of order $p$ in $N$ are representatives of the four conjugacy classes of nontrivial, proper, non self-normalizing subgroups of $G$. 
If $N$ is cyclic, then $G' \le \cent G N$ and $G' = N \times Q$ is abelian, so $G$ has derived length $2$.  Thus we may assume $N \cong C_p \times C_p$.  Notice that all of the subgroups of $N$ having order $p$ must lie in one $G$-conjugacy class, and that a representative of this conjugacy class can be chosen in the center of a Sylow $p$-subgroup $P$ of $G$; so, the size of the conjugacy class of subgroups of $N$ having order $p$ is $q$.  Since $N$ has $p+1$ subgroups of order $p$, we conclude that $q = p+1$, hence $p=2$ and $q=3$.  This implies that $P$ is isomorphic to the dihedral group of order $8$.  Note that $P$ has elements of order $2$ and of order $4$ that lie outside of $N$, and this gives two additional conjugacy classes of nontrivial, proper subgroups of $G$ that are not self-normalizing.  This contradiction completes the analysis for the case when $G'$ is not a $q$-group.

It remains to consider the situation in which $|G/G'| = p$ and $G'$ is a $q$-group (of order at most $q^4$). As above, $G'/N$ is an elementary abelian $q$-group of order $q^m$ for a suitable integer $m$.  Note that we can assume $N\neq 1$, as otherwise $G'$ is abelian and we are done; as a consequence, we can assume $m\leq 3$. Let $P$ be a Sylow $p$-subgroup of $G$, suppose $m\neq 1$, and suppose that $P$ has at least two orbits on the set of nontrivial proper subgroups of $G'/N$.  The two corresponding $G$-conjugacy classes of subgroups of $G'$, along with $G'$ and $N$, yield the four conjugacy classes of nontrivial, proper, non self-normalizing subgroups of $G$.  It follows that $N$ has order $q$.  Note that, as $P$ acts irreducibly on $G'/N$, $p$ does not divide $q^t-1$ for every integer $t$ that satisfies $0 < t < m$; in particular, $P$ centralizes $N$.  Now, setting $H = P\times N$ we see that $P$ is normal in $H$, hence $P$ is another non self-normalizing subgroup of $G$ and it produces a contradiction.

The above discussion excludes the case $m=3$, so $G'/N \cong C_q \times C_q$; moreover, $P$ acts transitively on the subgroups of order $q$ in $G'/N$, hence we have $p = q+1$ (i.e., $p = 3$ and $q = 2$). Since it can be checked via GAP that there do not exist groups of order $48$ with a nonabelian Sylow $2$-subgroup coinciding with the derived subgroup,  we have $|N|=2$. Moreover, $G'$ is a nonabelian subgroup of order $8$ that has an automorphism of order $3$. It is well known that $G'$ must be isomorphic to the quaternion group of order $8$, and hence $G$ is isomorphic to ${\rm SL}_2 (3)$.

The last possibility we have to consider is when $|G'/N| = q$. As usual, we can assume that $G'$ is nonabelian; thus, in this case, $N$ is noncyclic because otherwise we would get $G'\subseteq\cent G N$ and $G'$ would be central-by-cyclic (i.e., abelian). Now, observe that $\zent{G'}\cap N$ is nontrivial (because $G'$ is a $q$-group) and properly contained in $N$ (as otherwise $G'$ would be again central-by-cyclic), therefore $\zent{G'}\cap N$, $N$ and $G'$ are three distinct nontrivial, proper, normal subgroups of $G$. Also, take $x\in N-\zent{G'}$, $y\in G'-N$, and consider the subgroups $X=\langle x\rangle$, $Y=\langle y\rangle$. Since $X$ and $Y$ are proper subgroups of the $q$-group $G'$, certainly they are not self-normalizing. Moreover, it is clear that they yield two more conjugacy classes of nontrivial, proper, non self-normalizing subgroups of $G$. We reached the final contradiction, and the proof is complete.
\end{proof}

More generally we can bound the derived length of a solvable group in ${\mathcal D}_n$ by appealing to \cite[Theorem~8]{G}. 

\begin{theorem} \label{glasby}
Suppose $G$ is a solvable group that lies in ${\mathcal D}_n$. Then the derived length of $G$ is bounded by $3 \log_2(n +1) + 9$.
\end{theorem}

\begin{proof} 
As any proper subnormal subgroup is not self-normalizing it follows that if $G$ is solvable and lies in ${\mathcal D}_n$ then its composition length is bounded by $n+1$. The result now follows from \cite[Theorem 8]{G}.
\end{proof}

\section{Future work and open problems}

We would like to close this paper by listing some open problems and suggestions for future work.  In this paper, we have done a thorough study of $\mathcal{D}_0$, $\mathcal{D}_1$, $\mathcal{D}_2$, $\mathcal{D}_3$, and $\mathcal{D}_4$, and then, made several general observations coming from studying these cases.  In particular, we saw that for nilpotent groups in $\mathcal{D}_n$ the nilpotency class is at most $n/2$ and the derived length is at most $\log_2 (n/2) + 1$.  We also did  computations for Frobenius groups that led to the classifications that appear in Section \ref{frob sect}.  The general computation for the derived length of solvable groups is the following:

\begin{theorem}
Let $n \ge 3$ be an integer.  If $G \in \mathcal{D}_n$ is solvable, then the derived length of $G$ is less than or equal to the minimum of $n-1$ and $3 \log_2 (n+1)+9$.
\end{theorem}  

\begin{proof}
This is Proposition \ref{derivedlength} and Theorem \ref{glasby}.
\end{proof}

We have the solvability of all the groups in $\mathcal{D}_0$, $\mathcal{D}_1$, $\mathcal{D}_2$, $\mathcal{D}_3$, and $\mathcal{D}_4$ with the lone exception of ${\rm A}_5$.  Here are our suggested next steps:

\begin{enumerate}
\item Understand $\mathcal{D}_5$.  In particular, we think it is likely that all of the groups in this set are solvable, and it would be good to determine what derived lengths occur for solvable groups in this set.  From computations in GAP/Magma, we believe that they all have derived length at most $3$.  If this is correct, then we can argue that ${\rm dl} (G) \le n-2$ whenever $G$ is solvable and $G \in \mathcal{D}_n$ with $n \ge 5$.
\item What values of $n$ occur when $\mathcal{D}_n$ contains $G$ where $G$ is nonsolvable?  In particular, what is the next value of $n$ so that $\mathcal{D}_n$ contains a nonsolvable group?
\item For what values of $n$ (if any) does $\mathcal{D}_n$ contain nonisomorphic nonsolvable groups?
\item For what values of $n$ does $\mathcal{D}_n$ contain nonabelian simple groups?  
\item From the computations we have done in GAP/Magma, the formula $3 \log_2 (n+1)+9$ seems much bigger than the derived lengths that actually occur in $\mathcal{D}_n$, how much can we improve on this formula?
\end{enumerate}

The third author presented the results of this paper to the Urals Seminar on Group Theory and Combinatorics.  In response to this presentation, Gareth Jones has written a preprint \cite{JCount} which among other things counts the number of conjugacy classes of non self-normalizing subgroups of ${\rm PSL}_2(p)$ for primes $p$ that satisfy certain properties.  Based on his work, Professor Jones gets partial answers to our questions (2), (3), and (4); in particular, he shows in \cite[Theorem~1.4]{JCount} that there exists an infinite family of primes $p$ for which $\mathcal{D}({\rm PSL} (2,p))\leq 384$.

\end{document}